\documentclass[a4paper]{article}
\usepackage{epsf,algorithm,epsfig,amsfonts,amsgen,amsmath,amstext,amsbsy,amsopn,amsthm,lineno,amssymb}
\usepackage{color,xcolor}
\usepackage{tikz}
\usepackage{cite}
\usepackage{enumerate}
\usepackage{float}
\usepackage{epstopdf}
\allowdisplaybreaks[4]

\usepackage{amssymb}

\newtheorem{theorem}{Theorem}

\newtheorem{lemma}[theorem]{Lemma}

\newtheorem{claim}{Claim}

\newcounter{mathitem}

\usetikzlibrary{decorations.markings}
\tikzstyle{vertex}=[circle, draw, inner sep=0pt, minimum size=5pt]

\DeclareMathOperator{\ex}{ex}

\DeclareMathOperator{\rb}{rb}

\begin{document}

		\title{\bf\Large Rainbow matchings in edge-colored graphs \thanks{This research was supported by National Key Research and Development Program of China (No.~2023YFA1010203), and National Natural Science Foundation of China (Nos. ~12401464, 12471334, and 12271425).}}
	\date{}
	\author{
		Hongliang Lu$^{a}$,
		Zixuan Yang$^{b,c}$\thanks{Corresponding author. },
				Feihong Yuan$^{a}$		~\\[2mm]
		\small $^{a}$School of Mathematics and Statistics, \\ \small Xi'an Jiaotong University, Xi'an, Shaanxi 710049, P.R. China\\
		\small $^{b}$School of Mathematics and Statistics, \\
		\small Northwestern Polytechnical University, Xi'an, Shaanxi 710129, P.R. China\\
		\small $^{c}$Xi'an-Budapest Joint Research Center for Combinatorics, \\
		\small Northwestern Polytechnical University, Xi'an, Shaanxi 710129, P.R. China
		}
	\maketitle


\begin{abstract}
		Let $G$ be an edge-colored graph. We use $e(G)$ and $c(G)$ to denote the number of edges  and colors in $G$, respectively. A subgraph $H$ is called rainbow if $c(H)=e(H)$.    Li et al. (\emph{European J. Combin.}, \textbf{ 36} (2014), 453–459) proved that every edge-colored graph on $n$ vertices with $e(G)+c(G) \geq n(n+1)/2$ contains rainbow triangles. Later, Xu et al. (\emph{European J. Combin.}, \textbf{54 } (2016), 193–200) generalized
	the previous results concerning rainbow triangles to rainbow cliques $K_r$, where $r\geq 4$.  
	In this paper, we consider the existence of rainbow matchings of size $k$ in general edge-colored
	graphs $G$ under the condition of  $e(G)+c(G)$,   
	and  the condition in our result is tight.

	\smallskip
\end{abstract}

	\setcounter{footnote}{0}
\renewcommand{\thefootnote}{}
	\footnotetext{E-mail addresses:  {\tt luhongliang@mail.xjtu.edu.cn (H. Lu), yangzixuan@nwpu.edu.cn (Z. Yang), fhyuan@stu.xjtu.edu.cn (F. Yuan)}}


\section {\large Introduction}

\noindent All graphs considered in this paper are simple, finite, and undirected. For terminology and notation not defined herein, we refer to Bondy and Murty \cite{Bon}.

Given a graph $H$, let ex$(n, H)$ denote the maximum number of edges in an $n$-vertex graph that does not contain $H$ as a subgraph.   
In 1907,  Mantel  \cite{Man} determined the maximum number of edges in a triangle-free graph on
$n$ vertices is $\lfloor n^2/4\rfloor$. This result was later generalized by Turán's celebrated theorem \cite{Tur} in 1941. For $H = kK_2$, where $kK_2$ denotes a matching of size $k$, the value $\ex(n, kK_2)$ was established by Erdős and Gallai \cite{Erd2}.

\begin{theorem}[Erd\H{o}s and Gallai, \cite{Erd2}]\label{Erd2}
	
	For all $n \geq 2k$ and $k\geq 1$,
	\begin{align*}
		\ex(n,kK_2)=\max\Bigg\{\binom{k-1}{2}+(k-1)(n-k+1),\binom{2k-1}{2}\Bigg\}.
	\end{align*}
\end{theorem}

This paper focuses on the rainbow extensions of extremal problems. A subgraph of an edge-colored graph is called \emph{rainbow} if all its edges have distinct colors. For a fixed graph $H$, the \emph{rainbow number} $\rb(n, H)$ is defined as the minimum number of colors $r$ such that every edge-coloring of $K_n$ with $r$ colors contains a rainbow copy of $H$.  Equivalently, the \emph{anti-Ramsey number} $\text{ar}(n, H)$ is the maximum number of colors $r$ for which there exists an $r$-edge-coloring of $K_n$ with no rainbow $H$. It follows from these definitions that $\rb(n, H) = \text{ar}(n, H) + 1$. For consistency, we will use the rainbow number throughout this paper.

The study of anti-Ramsey theory was initiated by Erdős, Simonovits, and Sós \cite{Erd1},  who showed that  $\rb(n,K_r)=\ex(n,K_{r-1})+2$ for sufficiently large $n$.  This result was later extended to all $n > r \geq 3$ by Montellano-Ballesteros and Neumann-Lara \cite{Mon1} and independently by Schiermeyer \cite{Sch}. 
For the rainbow number of matchings, Schiermeyer \cite{Sch} proved that $\rb(n, kK_2) = \text{ex}(n, (k-1)K_2) + 2$ for $k \geq 3$ and $n \geq 3k + 3$.  Fujita et al. \cite{Fuj1} later obtained the same result for $n \geq 2k + 1$. Chen, Li, and Tu \cite{Che} eventually determined the exact values of $\rb(n, kK_2)$ for all $k \geq 3$ and $n \geq 2k$. Haas and Young \cite{Has} provided a simplified proof for the case $n = 2k$. For further results on rainbow numbers, we refer to the surveys by Kano and Li \cite{Kan} and Fujita, Magnant, and Ozeki \cite{Fuj}.

By considering both the number of edges $e(G)$ and the number of colors $c(G)$ in an edge-colored graph $G$, Li et al. \cite{Lib} established a rainbow version of Mantel’s theorem: if $e(G) + c(G) \geq \binom{n+1}{2}$, then $G$ contains a rainbow triangle.  This extends an earlier anti-Ramsey result of Erdős, Simonovits, and Sós \cite{Erd1}.  In 2019, Fujita et al. \cite{Fuj3} characterized all graphs $G$ satisfying $e(G) + c(G) \geq \binom{n+1}{2} - 1$ that contain no rainbow triangles. Ehard and Mohr \cite{Eha} further generalized the result of Li et al. \cite{Lib} by providing a sufficient condition on $e(G) + c(G)$ for the existence of $k$ rainbow triangles in $G$. The existence of rainbow triangles \cite{Aha,Cad,Czy} and the number of vertex-disjoint rainbow triangles \cite{Hu,Lih,Lix} in edge-colored graphs have also been extensively studied.

Following this line of research, Xu et al. \cite{Xu} extended results on rainbow triangles to larger rainbow cliques. They proved that if $e(G) + c(G) \geq \binom{n}{2} + \rb(n, K_r)$, then $G$ contains a rainbow $K_r$, generalizing the results of Montellano-Ballesteros and Neumann-Lara \cite{Mon1} and Schiermeyer \cite{Sch}. Recently, Liu, Peng, and Yang \cite{Liu} obtained a supersaturation version of the result of Xu et al. \cite{Xu}.

In this paper, we study the existence of rainbow matchings of size $k$ in general edge-colored graphs under a condition involving the sum of the number of edges and colors.

\begin{theorem}\label{main-thm}
	Let $k\geq 2$ and $n\ge 12$ be integers, and let $G$ be an edge-colored graph  on $n\geq 2k$ vertices. If
	\begin{align*}
		e(G)+c(G)\geq
			\binom{n}{2}+2+\max\Bigg\{\binom{2k-2}{2}-\binom{n-2k+2}{2},\ex(n,(k-1)K_2)\Bigg\},
				\end{align*}
	then $G$ contains a rainbow matching of size $k$.
\end{theorem}

We note that the condition on $e(G) + c(G)$ in Theorem \ref{main-thm} is best possible in a certain sense. Two natural constructions based on extremal graphs for matchings yield the lower bound $\binom{n}{2} + 2 + \mathrm{ex}(n, (k-1)K_2)$, while in the third construction, the host graph $G$ is obtained from $K_n$ by removing the edges of a clique. 

Let  $[n] := {1, 2, \ldots, n}$, and $\binom{[n]}{2} := { T \subseteq [n] : |T| = 2 }$.

\textbf{Construction \uppercase\expandafter{\romannumeral1}:}  $G$ is the complete graph on $[n]$. Color edges in $\left\{e\in\binom{[n]}{2}:e\cap [k-2]\ne \emptyset\right\}$ with $\binom{k-2}{2}+(k-2)(n-k+2)$ different colors and the others with a common new color.

\textbf{Construction \uppercase\expandafter{\romannumeral2}:}  $G$ is the complete graph on $[n]$. Color edges in $\binom{[2k-3]}{2}$ with $\binom{2k-3}{2}$ different colors, and the others with a common new color.

\textbf{Construction \uppercase\expandafter{\romannumeral3}:}  $G=\left\{e\in\binom{[n]}{2}:e\cap [2k-2]\ne \emptyset\right\}$. Color the edges in $\binom{[2k-2]}{2}$ with $\binom{2k-2}{2}$ different colors and color the others with a common new color.

\section{Preliminaries}

Let $G = (V(G), E(G))$ be a graph.
For a vertex $u\in V(G)$, its \emph{degree} is denoted by $d_G(u)$, and
the minimum vertex degree in $G$ is denoted by  $\delta(G)$. The neighborhood of $u$ is denoted by $N_G(u)$. For a set $S\subseteq V(G)$, define $N_G(S)=\bigcup_{u\in S}N_G(u)$.  For two disjoint subsets $S,T\subseteq V(G)$, let $E_{G}(S,T)$ denote the set of edges between $S$ and $T$, and write $e_G(S,T)=|E_G(S,T)|$.  For a subset $S\subseteq V(G)$ (or $F\subseteq E(G)$), the subgraph \emph{induced } by  $S$ (or $F$) is denoted by $G[S]$ (or $G[F]$). The graph obtained by deleting $S$ (or $F$) is denoted by $G-S$ (or $G-F$).  

The \emph{union} of two vertex-disjoint graphs $G$ and $H$, denoted by $G\cup H$, is defined by taking the union of their vertex sets and edge sets. The \emph{join}, denoted by $G \vee H$, is the graph obtained from $G \cup H$ by adding every possible edge between $V(G)$ and $V(H)$.

We use $K_n$, $K_{a,b}$, $P_n$, $S_n$, and $K_n^{-}$ to denote the complete graph on $n$ vertices, the complete bipartite graph with parts of sizes $a$ and $b$, the path on $n$ vertices, the star with $n$ edges (and $n+1$ vertices), and the graph obtained from $K_n$ by deleting one edge, respectively.

In an edge-colored graph $G$, the color of an edge $e\in E(G)$ is denoted by $C(e)$.  The set of edges in $G$ sharing the same color as $e$ is denoted by $EC(e)$.

A \emph{matching} in $G$ is a  set of disjoint edges.
A  \emph{maximum matching} is one of the largest possible size, and the \emph{matching number} $\nu(G)$ is the size of a maximum matching. A matching is \emph{perfect} if it covers all vertices of $G$. 
A graph $G$  is  \emph{factor-critical} if $G-\{v\}$  has a perfect matching for every $v\in V(G)$.  We will use the following fundamental result (see Exercise 3.3.18 in \cite{Lov}).

\begin{theorem}[Lov\'{a}sz and Plummer, \cite{Lov}]\label{GE-thm}
	Let $G$ be a graph without a perfect matchings. Then  there exists a subset  $S\subseteq V(G)$ such that:
	\begin{itemize}
		\item [\rm{(i)}] every component of $G-S$ is factor-critical;
		\item [\rm{(ii)}] every maximum matching  in $G$  matches each vertex of $S$ with vertices in different components of $G-S$.
	\end{itemize}
\end{theorem}

For positive integers $ a\geq b\geq 1$, and a graph $F$, define $ex (a ,b , F )$ as the maximum number of edges in an $F$-free subgraph of $K_{a,b}$. We will use the following result.

\begin{theorem}[Li, Tu, and Jin, \cite{Li}]\label{lem-Li}
	Let $a\geq b\geq k \geq 3$. Then
	$
	\ex(a,b,kK_2)=a(k-1).
	$
\end{theorem}

We also require the following classical result on Hamiltonian cycles.

\begin{theorem}[Dirac, \cite{Dir}]\label{Dir}
	
	Let $G$ be a graph on $n\geq 3$ vertices. If  $\delta(G)\geq \frac{n}{2}$, then $G$ is hamiltonian.
	
\end{theorem}

\section{Proof of Theorem \ref{main-thm}}

Let $G^{r}$ be a rainbow subgraph of $G$ obtained by selecting exactly one edge from each color class, so that $e(G^{r}) = c(G)$. For convenience,  define 
\begin{align*}
f_1(n,k):=\binom{2k-2}{2}-\binom{n-2k+2}{2}+2,
\end{align*}
\begin{align*}
f_2(n,k):=\binom{2k-3}{2}+2,
\end{align*}
 and 
 \begin{align*}
 f_3(n,k):=\binom{k-2}{2}+(n-k+2)(k-2)+2.
 \end{align*}

We proceed by contradiction. Suppose that $G^{r}$ does not contain a matching of size $k$. By Theorem~\ref{GE-thm} (i), there exists a subset $S \subseteq V(G^{r})$ such that every component of $G^{r} - S$ is factor-critical. Let $ |S|=s$, and let the components of $G^{r} - S$ be  $D_1, D_2, \cdots, D_q$ . For each $i \in [q]$, let $d_i = |V(D_i)|$ and fix an arbitrary vertex $u_i \in V(D_i)$. Without loss of generality, assume $d_1 \geq d_2 \geq \cdots \geq d_q \geq 1$.

Since $e(G)\leq \binom{n}{2}$, we have
\begin{align}\label{eq1}
	e(G^{r})=c(G)&\geq \binom{n}{2}+\max\{f_1(n,k),f_2(n,k),f_3(n,k)\}-e(G)\notag\\
	&\geq\max\{f_1(n,k),f_2(n,k),f_3(n,k)\}\notag\\
	&\ge \ex(n,(k-1)K_2)+2,
\end{align}
which implies that $G^{r}$ has a matching of size $k-1$, i.e., $\nu(G^{r})=k-1$.
By Theorem \ref{GE-thm} (i) and (ii), we obtain
\begin{align}\label{eq2}
	\nu(G^{r})=k-1=s+\sum_{i=1}^{q}\frac{d_i-1}{2},
\end{align}
\begin{align}\label{eq2'}
	|V(G^r)|=n=s+\sum_{i=1}^{q}d_i,
\end{align}
and
\begin{align}\label{eq3}
	e(G^{r})=e(G^r[S])+e_{G^r}(S,V(G^r)\setminus S)+\sum_{i=1}^qe(D_i)\leq \binom{s}{2}+(n-s)s+\sum_{i=1}^{q}\binom{d_i}{2}.
\end{align}

For three integers $a\geq b\geq c\geq 0$, we have the following inequality which will be used frequently in the sequel
\begin{align}\label{fact}
\binom{a}{2}+\binom{b}{2}\leq \binom{a+b-c}{2}+\binom{c}{2}.
\end{align}

The following proof  relies on Theorem~\ref{GE-thm} and is naturally divided into two cases ($n=2k$ and $n\ge 2k+1$), depending on whether the target rainbow matching is perfect.

\subsection{The Perfect Matching Case}

A direct calculation shows that $f_1(2k,k) = \max\{f_1(2k,k), f_2(2k,k), f_3(2k,k)\}$. Therefore, inequality (\ref{eq1}) implies
\begin{align}\label{eq4}
	e(G^r) \geq f_1(2k,k) = \binom{2k-2}{2} + 1.
\end{align}

\begin{claim}  \label{c4}
$0\leq s\leq 1$.
\end{claim}
\begin{proof}
By (\ref{eq2}), we suppose on the contrary that $2 \le s \le k-1$.
It follows from (\ref{eq3})  that
\begin{align*}
	e(G^r)&\leq\binom{s}{2}+\sum_{i=1}^{q}\binom{d_i}{2}+(n-s)s\\
	&\leq \binom{s}{2}+\binom{d_1+\sum_{i=2}^{q}(d_i-1)}{2}+(2k-s)s\quad(\mbox{by (\ref{fact})})\\
	&=\binom{s}{2}+(2k-s)s+\binom{2k-1-2s}{2}\quad(\mbox{by (\ref{eq2})})\\
	&=\frac{3s^2}{2}-2ks+\frac{5s}{2}+2k^2-3k+1=:f(s).
\end{align*}

By the convexity of $f$, it holds that
\begin{align*}
	e(G^r)\leq\max\{f(2),f(k-1)\}
	<2k^2-5k+4=\binom{2k-2}{2}+1,
\end{align*}
which contradicts  (\ref{eq4}). 
This completes the proof of Claim \ref{c4}.
\end{proof}
\begin{claim}\label{c5} 
$d_2=1$.
\end{claim}
\begin{proof}
By contradiction, suppose that $d_2 \geq 3$. From inequality (\ref{eq3}), we derive 
\begin{align}\label{eq5}
	e(G^r)&\leq\binom{s}{2}+(n-s)s+\binom{d_1}{2}+\binom{d_2}{2}+\sum_{j=3}^{q}\binom{d_j}{2}\notag\\
	&\leq \binom{s}{2}+(2k-s)s+\binom{d_1+(d_2-3)+\sum_{j=3}^{q}(d_j-1)}{2}+\binom{3}{2}\quad(\mbox{by (\ref{fact})})\notag\\
	&=\binom{s}{2}+(2k-s)s+\binom{2k-3-2s}{2}+3\quad(\mbox{by (\ref{eq2})}).
\end{align}
Note that $n=2k\ge 12$, and $s \in \left\{0, 1\right\}$ by Claim~\ref{c4}. Substituting into (\ref{eq5}) gives
\[
e(G^r) \leq \max\left\{2k^2 - 7k + 9,2k^2 - 9k + 17\right\} < \binom{2k - 2}{2} + 1,
\]
which contradicts (\ref{eq4}). This completes the proof of Claim~\ref{c5}.
\end{proof}
By Claim~\ref{c4}, we   consider  the following two cases.

\vspace{2mm}\textbf{Case 1.}~$s=0$.
\vspace{2mm}

By Claim~\ref{c5} and (\ref{eq2}), we have $d_2 = 1$ and $\nu(G^r) = k-1 = (d_1 - 1)/2$, which implies $d_1 = 2k-1$. 
Recall that $V(D_2)=\{u_2\}$. 

We first show that there exists a vertex $w_1 \in V(D_1)$ such that $u_2w_1 \in E(G)$. Otherwise, we have
\begin{align*}
	e(G)+e(G^r)\leq 2\binom{2k-1}{2}< \binom{2k}{2}+\binom{2k-2}{2}+1,
\end{align*}
contradicting (\ref{eq1}).

Next, we prove $e(G^r) \leq \binom{2k-2}{2} + 2$. Suppose the contrary. Consider the graph $H = G^r - \{w_1\} - EC(w_1 u_2)$ obtained by removing vertex $w_1$ and the edges of color $C(w_1 u_2)$. Then
\begin{align*}
	e(H)&\geq e(G^r)-d_{G^r}(w_1)-1\\
	&>\binom{2k-2}{2}+2-(2k-2)-1
	=ex(2k-2,(k-1)K_2),
\end{align*}
Hence, $H$ contains a rainbow matching $M_1$ of size $k-1$. Since $M_1$ avoids the color $C(w_1 u_2)$, the set $M_1 \cup \{w_1 u_2\}$ forms a rainbow matching of size $k$ in $G$, a contradiction. 

Combining $e(G^r) \leq \binom{2k-2}{2} + 2$ with  (\ref{eq1}), we obtain 
\begin{align*}
	 e(G)\geq\binom{2k}{2}+\binom{2k-2}{2}+1-e(G^{r})
	=\binom{2k}{2}-1,
\end{align*}
which implies that at most one vertex in $V(D_1)$ is not adjacent to $u_2$ in $G$. Let $W \subseteq V(D_1)$ be the set of vertices adjacent to $u_2$ in $G$, then $|W| \geq 2k-2$.

 We now show that there exists a vertex $w_2 \in W$ such that $d_{G^r}(w_2) \leq 2k-3$. Suppose not. Since $D_1$ is factor-critical, $\delta(D_1)\ge 2$. Then we have
\begin{align*}
	e(G^{r})&\ge \frac{1}{2}\sum_{v\in V(D_1)}d_{G^r}(v)
	\geq \frac{1}{2}((2k-2)(2k-2)+2)
	>\binom{2k-2}{2}+2,
\end{align*}
contradicting the upper bound on $e(G^r)$. Hence, such a vertex $w_2 \in W$ must exist.

Now fix $w_2 \in W$ with $d_{G^r}(w_2) \leq 2k-3$. We claim that $e(G^r) = \binom{2k-2}{2} + 1$. Otherwise,
\begin{align*}
	e(G^{r}-\{w_2\}-EC(u_2w_2))&\geq e(G^{r})-d_{G^{r}}(w_2)-1\\
	&\geq\binom{2k-2}{2}+2-(2k-3)-1\\
	&>2k^2-7k+6=ex(2k-2,(k-1)K_2),
\end{align*}
which implies the existence of a rainbow matching $M_2$ of size $k-1$ in $G^r - \{w_2\} - EC(u_2w_2)$ that avoids the color $C(u_2w_2)$. Then $M_2 \cup \{u_2w_2\}$ is a rainbow matching of size $k$ in $G$, a contradiction.

From $e(G^r) = \binom{2k-2}{2} + 1$ and (\ref{eq1}), we deduce that $e(G) = \binom{2k}{2}$, so $G \cong K_{2k}$. We next show that there exists a vertex $w_3 \in V(D_1)$ with $d_{G^r}(w_3) \leq 2k-4$. Suppose not; that is, $d_{G^r}(w) \geq 2k-3$ for all $w \in V(D_1)$. Then
\begin{align*}
	e(G^{r})&\geq \frac{1}{2}(2k-1)(2k-3)
	>\binom{2k-2}{2}+1,
\end{align*}
a contradiction. Therefore, such a vertex $w_3$ exists.

Finally, consider the graph $G^r - \{w_3\} - EC(w_3u_2)$. We have
\begin{align*}
	e(G^{r}-\{w_3\}-EC(w_3u_2))&\geq e(G^{r})-d_{G^{r}}(w_3)-1\\
	&\geq\binom{2k-2}{2}+1-(2k-4)-1\\
	&>2k^2-7k+6=ex(2k-2,(k-1)K_2),
\end{align*}
so there exists a rainbow matching $M_3$ of size $k-1$ in this graph avoiding color $C(u_2w_3)$. Then $M_3 \cup \{u_2w_3\}$ is a rainbow matching of size $k$ in $G$, the final contradiction.

\vspace{2mm}\textbf{Case 2.}~$s=1$.
\vspace{2mm}

By Claim~\ref{c5} and (\ref{eq2}), we have $d_2 = 1$ and $\nu(G^r) = k - 1 = 1 + (d_1 - 1)/2$, which implies $d_1 = 2k - 3$, and then by (\ref{eq2'}), $d_3=1$ and $d_4=0$. Recall that $V(D_2) = \{u_2\}$, and $V(D_3) = \{u_3\}$. Let $S=\{v\}$. Then by (\ref{eq4}), we obtain
\begin{align}\label{2.2}
\binom{2k-2}{2}+1\leq e(G^{r})=e(D_1)+e_{G^r}(\{v\},V(D_1))+e_{G^r}(\{v\},\{u_2,u_3\})\leq\binom{2k-2}{2}+2.
\end{align}

 We first prove that $u_2 u_3 \notin E(G)$. Suppose, for contradiction, that $u_2 u_3 \in E(G)$. Then
\begin{align*}
	e(G^{r}-\{u_2,u_3\}-EC(u_2u_3))&\geq e(G^{r})-d_{G^{r}}(u_2)-d_{G^{r}}(u_3)-1\\
	&\geq\binom{2k-2}{2}+1-1-1-1\\
	&>ex(2k-2,(k-1)K_2),
\end{align*}
so $G^r - \{u_2, u_3\} - EC(u_2 u_3)$ contains a matching $M_4$ of size $k - 1$ that avoids color $C(u_2 u_3)$. Thus, $M_4 \cup \{u_2 u_3\}$ forms a rainbow matching of size $k$ in $G$, a contradiction.

Since $u_2 u_3 \notin E(G)$,  $e(G) \leq \binom{2k}{2} - 1$. Together with (\ref{2.2}) and (\ref{eq1}), this gives
\[ e(G)=\binom{2k}{2}-1\ \ \  \text{and}\ \ \  e(G^r)=\binom{2k-2}{2}+2.\]

We now show that for any $w \in V(D_1)$,
$C(u_2w)=C(vu_3)$ and $C(u_3w)=C(vu_2)$.
 Suppose, for some $w \in V(D_1)$, that $c(u_2 w) \ne c(v u_3)$. Then
\begin{align*}
e(G^{r}-\{v,w\}-EC(u_2w))&\geq e(G^r)-d_{G^r}(v)-(d_{G^r}(w)-1)-1\\
&\geq \binom{2k-2}{2}+2-(2k-1)-(2k-4)-1\\
&>ex(2k-4,(k-2)K_2),
\end{align*}
which implies the existence of a matching $M_5$ of size $k - 2$ in $G^r - \{v, w\} - EC(u_2 w)$ avoiding color $C(u_2 w)$. Thus, $M_5 \cup \{v u_3, u_2 w\}$ is a rainbow matching of size $k$ in $G$, a contradiction. A similar contradiction is reached by symmetry if $c(u_3 w) \ne c(v u_2)$. 

Now, fix $w_4, w_5 \in V(D_1)$ such that
$C(u_2w_4)=C(vu_3)$ and $C(u_3w_5)=C(vu_2)$.
Then
\begin{align*}
e(G^{r}-\{w_4,w_5,u_2,u_3\}-EC(u_2w_4,u_3w_5))
&\geq\binom{2k-2}{2}+2-(2k-3)-(2k-4)-2-2\\
&>ex(2k-4,(k-2)K_2),
\end{align*}
so there exists a rainbow matching $M_6$ of size $k - 2$ in $G^{r}-\{w_4,w_5,u_2,u_3\}$ that avoids colors $C(u_2 w_4)$ and $C(u_3 w_5)$.
Thus, $M_6 \cup \{u_2 w_4, u_3 w_5\}$ is a rainbow matching of size $k$ in $G$, the final contradiction.

\subsection{The Non-Perfect Matching Case}

We will use the following two lemmas. The first lemma is due to Chen, Li, and Tu~\cite{Che}.

\begin{lemma}[Chen, Li, and Tu, Lemma 3.5 in \cite{Che}]\label{lemma2}
	Let $G$ be the complete graph, and suppose  $G^r\in \{K_{2k-3}^{-}\cup K_3\cup \overline{K_{n-2k}}, K_{2k-3}\cup P_3\cup \overline{K_{n-2k}}\}$. Then $G$ contains a rainbow matching of size $k$.
\end{lemma}

\begin{lemma}\label{lemma1}
	Let $G$ be a graph obtained from $K_n$ by deleting at most one edge, and suppose   $	G^r\cong K_{2k-3}\cup K_3\cup \overline{K_{n-2k}}$. Then $G$ contains a rainbow matching of size $k$.
\end{lemma}
\begin{proof}

	Let $V(D_1)=\left\{w_1,\dots,w_{2k-3}\right\}$, $V(D_2)=\left\{w'_1,w'_2,w'_3\right\}$, and $U=\left\{u_3,\ldots,u_{n-2k+2}\right\}$ denote the vertex sets of the components $K_{2k-3}$, $K_3$,  and $\overline{K_{n-2k}}$ in  $G^r$, respectively.  
There exists a matching of size 3 in $G$ between $\left\{w_1,w_2,w_3\right\}$ and $V(D_2)$, and a vertex  $w_i\in V(D_1)$ such that $u_3w_i\in E(G)$, where $i\in [2k-3]$. Without loss of generality, assume $w_1w'_1,w_2w'_2,w_3w'_3\in E(G)$. 
	
	First, we consider the case $k=3$. If $\left\{w_1w'_1,w_2w_3,w'_2w'_3\right\}$ is not rainbow , then $C(w_1w'_1)\in \left\{C(w_2w_3),C(w'_2w'_3)\right\}$. By a similar argument, $C(w_2w'_2)\in \left\{C(w_1w_3),C(w'_1w'_3)\right\}$ and $C(w_3w'_3)\in \left\{C(w_1w_2),C(w'_1w'_2)\right\}$. It follows that $\left\{w_1w'_1,w_2w'_2,w_3w'_3\right\}$ is a rainbow matching, as desired.
	
	Next, we consider the case $k\ge 4$.  If the edge colored by $C(w_iu_3)$ in $G^r$ is  contained in $E(D_2) \cup E_{G^r}(\{w_i\}, V(D_1))$, then there exist  matchings of size $k-2$ and $1$ in $D_1-\{w_i\}$ and $D_2$ respectively that do not contain $C(w_iu_3)$. This implies that there is a rainbow matching of size $k$ in $G$, which is a contradiction. So we assume that  $EC(w_iu_3)\cap E(G^r) \in E(D_1-\{w_i\})$. Write $D'_1=D_1-\{w_i\}-EC(w_iu_3)$. Note that $\delta(D')= 2k-6\ge |V(D')|/2$. By Theorem \ref{Dir}, there exists a matching $M_2$ of size $k-2$ in $D'$ that does not contain $C(w_iu_3)$. Thus, $M_2\cup \{w'_1w'_2,w_iu_3\}$ is a rainbow matching of size $k$ in $G$, again a contradiction.
	\end{proof}	
	

By  (\ref{eq2}), we have $0\le s\le k-1$. The proof next splits naturally into three cases according to the value of 
$s$.

\vspace{2mm}\textbf{Case 1.}~$s=k-1$.
\vspace{2mm}

By (\ref{eq2}) and (\ref{eq2'}), we have $d_1=1$ and  $q=n-k+1$. Recall that $V(D_i)=\{u_i\}$ for  $i=1,2,\ldots, q$. 

\begin{claim}\label{c1}
	For every $i\in[q]$, there exists some $j\in [q]$ (with $j\ne i$) such that $u_iu_j\in E(G)$.
\end{claim} 
\begin{proof}

Suppose, by contradiction, that there exists a vertex  $u_i$ which is not adjacent to any other vertex $u_j$ in $G$. Then we have $e(G)\leq\binom{n}{2}-(n-k)$. Combining this with (\ref{eq3}), we derive
\begin{align*}
	e(G)+e(G^{r})&\leq  \binom{n}{2}-(n-k) + \binom{k-1}{2}+(n-k+1)(k-1)<\binom{n}{2}+f_3(n,k),
\end{align*}
which contradicts (\ref{eq1}). Therefore, the claim holds.
\end{proof}

Let $H^{r}_1$ be the bipartite graph obtained from $G^{r}$ by removing all edges in  $G^r[S]$.

\begin{claim}\label{c2}
	For any distinct  $i,j\in[q]$, if $u_iu_j\in E(G)$, 
	then $$e(H^{r}_{1}-\{u_i,u_j\})\leq (n-k-1)(k-2)+1.$$
\end{claim}
\begin{proof}

Assume $u_iu_j\in E(G)$ and suppose, for contradiction, that   $e(H^{r}_{1}-\{u_i,u_j\})> (n-k-1)(k-2)+1$. Then
\begin{align*}
	e(H^{r}_{1}-\{u_i,u_j\}-EC(u_iu_j))>(n-k-1)(k-2)=\ex(n-k-1,k-1,(k-1)K_2)
\end{align*}
where the equality follows from Theorem \ref{lem-Li}. This implies that $H^{r}_{1}-\{u_i,u_j\}-EC(u_iu_j)$ contains a rainbow matching $M_3$ of size $k-1$ that avoids the color $C(u_iu_j)$. Thus, $M_3\cup \{u_iu_j\}$ forms a rainbow matching of size $k$ in $G$, a contradiction.
\end{proof}

\begin{claim}\label{c3}
	There exist  two integers $l$ and $p$ with $1\le l<p\le q$ such that $d_{G^{r}}(u_l)\leq k-2$ and $d_{G^{r}}(u_p)\leq k-2$.
\end{claim}
\begin{proof}

Suppose to the contrary that at most one vertex $u_i$ satisfies $d_{G^{r}}(u_i)\leq k-2$. Without loss of generality, assume   $d_{G^{r}}(u_i)= k-1$ for all $i= 2,\ldots,q$. By Claim \ref{c1}, there exists $l\in \{2,\ldots,q\}$ such that $u_1u_l\in E(G)$. Then 
\begin{align*}
	e(H^{r}_{1}-\{u_1,u_l\})=(n-k-1)(k-1)>(n-k-1)(k-2)+1,
\end{align*}
contradicting Claim \ref{c2}. This completes the proof of Claim \ref{c3}.
\end{proof}

By Claim~\ref{c3}, without loss of generality, we may assume $d_{G^r}(u_1)\leq k-2$ and $d_{G^r}(u_2)\leq k-2$. 
First, if $u_1u_2\in E(G)$, then from (\ref{eq1}),  we have
\begin{align*}
	e(G^r)
	&= d_{G^r}(u_1)+d_{G^r}(u_2)+e(G^r[S])+e(H^{r}_{1}-\{u_1,u_2\})\\
	&\leq 2(k-2)+\binom{k-1}{2}+(n-k-1)(k-2)+1\quad(\mbox{by Claim \ref{c2}})\\
	&=\binom{k-2}{2}+(n-k+2)(k-2)+1,
\end{align*}
which contradicts the lower bound $e(G^r)\ge f_3(n,k) $. 

So we may assume $u_1u_2\notin E(G)$, and then $e(G)\leq \binom{n}{2}-1$. By Claim \ref{c1}, there exists $i\in\{3,\dots,q\}$ such that $u_1u_i\in E(G)$. Thus, 
\begin{align*}
	e(G^r)&= d_{G^r}(u_1)+d_{G^r}(u_i)+e(G^r[S])+e(H^{r}_{1}-\{u_1,u_i\})\\
	&\leq (k-2)+(k-1)+\binom{k-1}{2}+(n-k-1)(k-2)+1\quad(\mbox{by Claim \ref{c2}})\\
	&=\binom{k-2}{2}+(n-k+2)(k-2)+2,
\end{align*}
contradicting the fact that $e(G^r)\ge f_3(n,k)+1$.

\vspace{2mm}\textbf{Case 2.}~$s=0$.
\vspace{2mm}

\begin{claim}\label{c7}
 $d_2\leq 3$ and $d_3=1$.
\end{claim}
\begin{proof}
By contradiction. Suppose that  $d_2\geq 5$ or $d_3\ge 3$. We first consider $d_2\ge5$. Then by (\ref{eq2}), $k\ge (d_1+d_2)/2\ge 5$. By (\ref{eq3}), we have
\begin{align*}
e(G^r)\leq\sum_{i=1}^{q}\binom{d_i}{2}&\leq \binom{d_1+(d_2-5)+\sum_{i=3}^{q}(d_i-1)}{2}+\binom{5}{2}\quad(\mbox{by (\ref{fact})})\\
&\leq\binom{2k-5}{2}+10\quad(\mbox{by (\ref{eq2})})\\
&< \binom{2k-3}{2}+2\quad(\mbox{since $k\ge 5$}),
\end{align*}
contradicting inequality (\ref{eq1}). Hence, $d_2\leq 3$. 

Next, we consider $d_3\geq 3$. Then (\ref{eq2}) implies $k\ge 4$. By (\ref{eq3}),  we have
\begin{align*}
e(G^r)\leq\sum_{i=1}^{q}\binom{d_i}{2}&\leq \binom{d_1+(d_2-3)+(d_3-3)+\sum_{i=4}^{q}(d_i-1)}{2}+2\binom{3}{2}\quad(\mbox{by (\ref{fact})})\\
&\leq\binom{2k-5}{2}+6\quad(\mbox{by (\ref{eq2})})\\
&< \binom{2k-3}{2}+2\quad(\mbox{since $k\ge 4$}),
\end{align*}
again contradicting (\ref{eq1}). Therefore, $d_3=1$.
\end{proof}

We now consider two subcases according to the value of $d_2$.

\vspace{2mm}\textbf{Subcase 2.1.}~$d_2=3$.
\vspace{2mm}

By Claim~\ref{c7}, $d_3=1$, and by (\ref{eq2}) we have $d_1=2k-3\geq 3$. Then by (\ref{eq1}),  we have $\binom{2k-3}{2}+2\leq e(G^r)\leq \binom{2k-3}{2}+3$.
So
\begin{align*}
G^r\in \{K_{2k-3}\cup K_3\cup \overline{K_{n-2k}},K_{2k-3}^{-}\cup K_3\cup \overline{K_{n-2k}}\},
\end{align*}
and consequently,
\begin{align*}
	G\in \{K_{n},K_{n}^{-}\}.
\end{align*}
By Lemmas \ref{lemma2} and \ref{lemma1}, $G$ contains a rainbow matching of size $k$, a contradiction.

\vspace{2mm}\textbf{Subcase 2.2.}~$d_2=1$.
\vspace{2mm}

By (\ref{eq2}), we have $d_1=2k-1$. Then by (\ref{eq2'}), $q=n-2k+2$. Let $D=\bigcup_{i=2}^{q}\{u_i\}$, where $V(D_i)=\{u_i\}$ for $i=2,\dots,q$. 

We first show $E(G[D])= \emptyset$.  Suppose, for contradiction, that $E(G[D])\neq \emptyset$, and let $u_iu_j\in E(G)$ for some $i,j\in \{2,\ldots,q\}$. Choose an edge $xy\in E(G^r)$ with color $C(u_iu_j)$. Since  $D_1$ is factor-critical and has order $2k-1$, the subgraph $D_1-\{x\}$ has a perfect matching $ M_4$ of size $k-1$ that avoids the color $C(u_iu_j)$. Then $M_4\cup \{u_iu_j\}$ forms a rainbow matching of size $k$ in $G$, a contradiction. Therefore,  $E(G[D])= \emptyset$.

\begin{claim}\label{c10} $E_{G}(D,V(D_1))\neq \emptyset$.
\end{claim}
\begin{proof}

Since $e(G^r)=e(D_1)\le \binom{2k-1}{2}$ and $e(G[D])=0$,  we have
\begin{align*}
	e(G)+e(G^r)\le 2\binom{2k-1}{2}+E_{G}(D,V(D_1)).
\end{align*}
By the theorem hypothesis, $e(G)+e(G^r)\ge \binom{n}{2}+f_1(n,k)$. Comparing these bounds yields
 $e_{G}(D,V(D_1))>0$, so $E_{G}(D,V(D_1))\neq \emptyset$.
\end{proof}
\begin{claim} \label{c11}
$\delta(G^r-D)\leq k-1$.
\end{claim}
\begin{proof}
Suppose, for contradiction, that 
 $\delta(G^r-D)\geq k$. By Claim \ref{c10}, there exists an edge  $u_iw\in E(G)$ with $w\in V(D_1)$ and $2\leq i\leq q$. Let $G'=G^r-(D\cup \{w\})$. Then $|V(G')|=2k-2$ and  $\delta(G')\geq k-1=|V(G')|/2$. By Theorem \ref{Dir}, $G'$ contains a Hamilton cycle, which implies the existence of two edge-disjoint perfect matchings $M_5$ and $M_6$, each of size $k-1$,with disjoint color sets. At least one of them, say $M_5$, avoids color $C(u_iw)$. Then $M_5\cup \{u_iw\}$ is a rainbow matching of size $k$ in $G$, a contradiction.
\end{proof}

\begin{claim}\label{c12}  For any vertex $w\in V(D_1)$, if there exists $i\in\{2,\dots,q\}$ such that $wu_i\in E(G)$, 
then $$
e(G^r-\{w\})\leq \binom{2k-3}{2}+1.$$
\end{claim}
\begin{proof}

By contradiction. Suppose that there exists a vertex $w\in V(D_1)$ and an index $i\in\{2,\dots,q\}$ with $wu_i\in E(G)$, such that $e(G^r-\{w\})> \binom{2k-3}{2}+1$. Then
\begin{align*}
e(G^r-\{w\}-EC(u_iw))> \binom{2k-3}{2}=ex(2k-2,(k-1)K_2).
\end{align*}
This implies that there exists a rainbow matching $M_7$ of size $k-1$ in $G^r-\{w\}-EC(u_iw)$  that avoids color $C(u_iw)$. Thus, $M_7\cup \{u_iw\}$ forms a rainbow matching of size $k$ in $G$, a contradiction.
\end{proof}

Now we derive the final contradiction by calculating the sum of $e(G)$ and $e(G^r)$. By Claim \ref{c11}, there exists a vertex $w\in V(D_1)$ with $d_{G^r}(w)\leq k-1$. 

We first show that $u_iw\notin E(G)$ for all $i\in\{2,\dots,q\}$. Otherwise, suppose  that $u_iw\in E(G)$ for some $i\in\{2,\dots,q\}$. Then by Claim~\ref{c12}, we have
\begin{align*}
	e(G^r)+e(G)
	&= e(G^r-\{w\})+d_{G^r}(w)+e(G)\\
	&\leq \binom{2k-3}{2}+1+(k-1)+\binom{n}{2}-\binom{n-2k+1}{2}\\
	&<\max\{f_1(n,k),f_3(n,k)\}+\binom{n}{2},
\end{align*}
contradicting (\ref{eq1}). Hence, $u_iw\notin E(G)$ for all $i=2,\dots,q $.

 Since $d_{G_r}(w)\leq k-1$ and $|V(G^r)\setminus (D\cup \{w\})|=2k-2\geq k$, there exists  $w'\in V(D_1)\setminus\{w\}$ such that $ww'\notin E(G^r)$, which implies $d_{G^r}(w')\leq 2k-3$. 

We now show that   $w'u_i\notin E(G)$ for all $i\in\{2,\dots,q\}$. If $w'u_i\in E(G)$ for some $2\le i\le q$,   then by Claim~\ref{c12} and the fact that $u_iw\notin E(G)$ for all $i\in\{2,\dots,q\}$, we have
\begin{align*}
e(G^r)+e(G)
&\leq e(G^r-\{w'\})+d_{G^r}(w')+e(G)\\
&\leq \binom{2k-3}{2}+1+(2k-3)+\binom{2k-1}{2}+(n-2k+1)(2k-2)\\
&<\binom{2k-2}{2}-\binom{n-2k+2}{2}+2+\binom{n}{2},
\end{align*}
 contradicting (\ref{eq1}). Therefore, $w'u_i\notin E(G)$ for all $i\in\{2,\dots,q\}$.

 By Claim \ref{c10}, there exist $w''\in V(D_1)$ and $u_j\in D$ with $w''u_j\in E(G)$. Then by Claim~\ref{c12} and the fact that  $u_iw, u_iw'\notin E(G)$  for all $i\in\{2,\dots,q\}$, we have
 \begin{align*}
e(G^r)+e(G)
&\leq e(G^r-\{w''\})+d_{G^r}(w'')+e(G)\\
&\leq \binom{2k-3}{2}+1+(2k-2)+\binom{2k-1}{2}+(n-2k+1)(2k-3)\\
&<\binom{2k-2}{2}-\binom{n-2k+2}{2}+2+\binom{n}{2},
\end{align*}
 contradicting (\ref{eq1}).

\vspace{2mm}
\textbf{Case 3.}~$1\leq s\leq k-2$.
\vspace{2mm}

\begin{claim}\label{c14}  
 $d_2=1$.
\end{claim}
\begin{proof}
By contradiction, suppose that $d_2\geq 3$. Since $d_1\geq d_2$, (\ref{eq2}) implies $s\leq k-3$. From (\ref{eq3}) and (\ref{eq2}), we derive 
\begin{align*}
e(G^r)&\leq\binom{s}{2}+\sum_{i=1}^{q}\binom{d_i}{2}+(n-s)s\\
&\leq \binom{s}{2}+(n-s)s+\binom{d_1+(d_2-3)+\sum_{i=3}^{q}(d_i-1)}{2}+\binom{3}{2}\\
&\leq\binom{s}{2}+(n-s)s+\binom{2k-3-2s}{2}+3.
\end{align*}
For $1\le s \le k-3$, the right-hand side is strictly less than $\max\{f_2(n,k),f_3(n,k)\}$,
contradicting (\ref{eq1}).
\end{proof}

By (\ref{eq2}) and Claim \ref{c14}, we have $d_1=2k-1-2s$. Define $D=\bigcup_{i=2}^{q}\{u_i\}$, where $V(D_i)=\{u_i\}$ for each $i\in\{2,\dots,q\}$ and $q=n-2k+s+2$.

\vspace{2mm}\textbf{Subcase 3.1.}~$E(G[D])=\emptyset$.
\vspace{2mm}

We first observe that
\begin{align}\label{eq5'}
e(G)\leq \binom{n}{2}-\binom{|D|}{2}=\binom{n}{2}-\binom{n-2k+1 +s}{2}.
\end{align}




We now prove that $s=1$. Suppose, to the contrary, that $2\le s\le k-2$. Let 
\begin{align*}
g(s,n):&=-\binom{n-2k+1 +s}{2}+\binom{s}{2}+\binom{2k-1-2s}{2}+(n-s)s\\
&=-\frac{n^2}{2}+n(2k-\frac{1}{2})+s^2+s(2-2k)+(1-2k).
\end{align*}
One can see that $g(s,n)\le g(2,n)\le g(2,2k+1)$.
Then by (\ref{eq5'}) and (\ref{eq3}), we have
\begin{align*}
	e(G)+e(G^r)&\le \binom{n}{2}+g(s,n)\le  \binom{n}{2}+g(2,2k+1)<\max\{f_1(n,k),f_2(n,k)\}+\binom{n}{2},
	\end{align*}
which contradicts (\ref{eq1}). Hence, we  have $s=1$.

It follows that  $d_1=2k-3$ and $|D|=n-2k+2\geq 3$. Let $S=\{v\}$. From (\ref{eq1}), (\ref{eq3}) and (\ref{eq5'}), we obtain
\begin{align}\label{eq10}
\binom{2k-3}{2}+(n-1)\geq e(G^r)&\geq \max\{f_1(n,k),f_2(n,k)\}+\binom{n-2k+2}{2},
\end{align}
and 
\begin{align}\label{eq11'}
\binom{n}{2}-\binom{n-2k+2}{2}&\geq e(G)\geq  \binom{n}{2}-n+3.
\end{align}
Note that 
$$e_{G^r}(D,\{v\})=e(G^r)-e(G^r[V(D_1)\cup \{v\}])\geq \binom{2k-2}{2}+2-\binom{2k-2}{2}=2.$$ 
Hence, without loss of generality, we may assume that $u_2v, u_3v\in E(G^r)$.  

We claim that $E_G(\{u_i\},V(D_1))\neq \emptyset$ for every  $i\in\{2,\dots,q\}$. Indeed, we have
\begin{align*}
	e_G(D,V(D_1))&=e(G)-e(D_1)-d_G(v)\\
	&\ge \binom{n}{2}-n+3-\binom{2k-3}{2}-(n-1)\quad(\mbox{by (\ref{eq11'})})\\
	&>(n-2k+1)(2k-3),
\end{align*}
which confirms the claim.

Let $w\in V(D_1)$ be such that $wu_4\in E(G)$.  Since $D_1$ is factor-critical, the graph $D_1-\{w\}$ contains a matching $M_8$ of size $k-2$. As neither $M_8\cup \{vu_2,u_4w\}$ nor $M_8\cup \{vu_3,u_4w\}$ is rainbow,  there must exist an edge in $M_8$ sharing the color  $C(u_4w)$. 

By Theorem \ref{Erd2}, we have
\begin{align*}
e(D_1-\{w\}-EC(u_4w))\leq ex(2k-4,(k-2)K_2)=\binom{2k-5}{2}.
\end{align*}
Otherwise, $D_1-\{w\}-EC(u_4w)$ would contain a rainbow matching $M_9$  of size $k-2$, in which case either $M_9\cup \{vu_2,u_4w\}$ or $M_9\cup \{vu_3,u_4w\}$ would be a rainbow matching of size $k$, a contradiction.

Therefore, we obtain
\begin{align*}
e(G^r)&\leq e(D_1-\{w\}-EC(u_4w))+1+d_{D_1}(w)+e_{G^r}(\{v\},V(G^r)\setminus\{v\})\notag\\
&\leq \binom{2k-5}{2}+1+(2k-4)+(n-1)\notag\\
&=\binom{2k-3}{2} +(n-1)-(2k-6).
\end{align*}
For $n\ge 2k+1$ and $n\ge 12$, this contradicts the lower bound in (\ref{eq10}).

\vspace{2mm}\textbf{Subcase 3.2.}~$E(G[D])\neq\emptyset$.
\vspace{2mm}

Fix vertices $u_2,u_3\in D$ with $u_2u_3\in E(G)$. Define 
 $H^{r}_{2}$ as the bipartite graph obtained from $G^r$ by removing $D_1$ and all edges in $G^r[S]$.

\begin{claim}\label{c18} 
	For any $i,j\in \{2,\dots,q\}$, if $u_iu_j\in E(G)$,  then $H^{r}_{2}-\{u_i,u_j\}$ does not contain a matching of size $s$.
\end{claim}
\begin{proof}

By contradiction. Suppose that there exist $i,j\in \{2,\dots,q\}$ with $u_iu_j\in E(G)$ and a matching $M$  of size $s$ in $H^{r}_{2}-\{u_i,u_j\}$. Let  $uv\in E(G^r)$ be the edge colored $C(u_iu_j)$. Since $D_1$ is factor-critical and has order $2k-1-2s$, the graph $D_1-\{uv\}$ contains a matching $M_1$ of size $k-1-s$ that avoids color $C(uv)$.

If $uv\notin M$, then $M_1\cup M\cup \{u_iu_j\}$ is a rainbow matching of size $k$ in $G$, a contradiction.
Hence, $uv\in M$, and without loss of generality, we let  $v\in S$. 

We now show that
\begin{align}\label{eq7}
E_{G^r}(\{v\},V(D_1))=\emptyset.
\end{align}
Suppose otherwise, and let $w\in V(D_1)$ with $vw\in E(G^r)$. Then $D_1-\{w\}$ contains a matching $M_2$ of size $k-1-s$, and $M_2\cup (M\setminus \{uv\})\cup \{vw,u_iu_j\}$ forms a rainbow matching of size $k$ in $G$, again a contradiction. This establishes (\ref{eq7}).

Next, by Theorem~\ref{lem-Li}, we have
\begin{align}\label{eq8}
e(H^{r}_{2}-\{u_i,u_j\})\leq ex(n+s-2k-1,s,sK_2)+1=(n+s-2k-1)(s-1)+1.
\end{align}
Indeed, if this bound is violated, then $H^{r}_{2}-\{u_i,u_j\}-\{uv\}$ would contain a matching $M_3$ of size $s$. Together with a matching $M_4$ of size $k-1-s$ from $D_1$, $M_3\cup M_4\cup \{u_iu_j\}$ would be a rainbow matching of size $k$, a contradiction.

Now, combining (\ref{eq7}) and (\ref{eq8}), we obtain
\begin{align}\label{eq9}
e(G^r)&= e(G^r[S])+e(D_1)+e_{G^r}(V(D_1),S)+d_{G^r}(u_i)+d_{G^r}(u_j)+e(H^{r}_{2}-\{u_i,u_j\})\notag\\
&\leq \binom{s}{2}+\binom{2k-2s-1}{2}+(2k-2s-1)(s-1)+2s+(n+s-2k-1)(s-1)+1\notag\\
&=\frac{3s^2}{2}+s(n-4k+\frac{7}{2})-n+2k^2-3k+4=:h(s).
\end{align}

If  $2\leq s\leq k-2$, then $e(G^r)\leq \max\{h(2), h(k-2)\}\leq  ex(n,(k-1)K_2)+1$, contradicting inequality (\ref{eq1}).

Therefore, we  have $s=1$. Let $S=\{v\}$. Then from (\ref{eq1}) and (\ref{eq9}),
\begin{align*}
\binom{2k-3}{2}+2\leq e(G^r)\le  h(1)=\binom{2k-3}{2}+3.
\end{align*}
For  $e(G^r)=\binom{2k-3}{2}+2$,  we see that $G^r\in \{ K_{2k-3}\cup P_3\cup \overline{K_{n-2k}},K_{2k-3}^{-}\cup S_3\cup \overline{K_{n-2k-1}}\}$; for $ e(G^r)=\binom{2k-3}{2}+3$,  we have $G^r\cong K_{2k-3}\cup S_3\cup \overline{K_{n-2k-1}}$.
If $G^r$ is  $K_{2k-3}\cup P_3\cup \overline{K_{n-2k}}$, then by Lemma \ref{lemma2}, we have $G$ contains a rainbow matching of size $k$, which is a contradiction.
If $G^r\in \{K_{2k-3}^{-}\cup S_3\cup \overline{K_{n-2k-1}},K_{2k-3}\cup S_3\cup \overline{K_{n-2k-1}}\}$, then we consider a new graph $G^{r'}:=G^r-\{uv\}+\{u_iu_j\}$. Since $C(uv)=C(u_iu_j)$, the graph $G^{r'}$ is a rainbow  subgraph of $G$ and  $G^{r'}\in \{ K_{2k-3}\cup K_3\cup \overline{K_{n-2k}},K_{2k-3}^{-}\cup K_3\cup \overline{K_{n-2k}}\}$. Applying Lemmas \ref{lemma2} and  \ref{lemma1},  we again obtain a contradiction. 
\end{proof}

\begin{claim}\label{c19} 
 $2\leq s\leq k-2$.
\end{claim}
\begin{proof}

Suppose, for contradiction, that $s=1$. Let $S=\{v\}$. Recall that $d_1=2k-1-2s=2k-3$ and $D=\bigcup_{i=2}^{n-2k+3}\{u_i\}$.  By Claim \ref{c18}, we have $e(H_{2}^{r}-\{u_2,u_3\})=0$, which implies  $1\le e(H_{2}^{r})\le 2$.  

If $e(H_{2}^{r})=1$ assume without loss of generality that  $u_2v\in E(G^r)$ and $u_3v\notin E(G^r)$. By Claim \ref{c18},  $u_4u_3\notin E(G)$, so $e(G)\leq \binom{n}{2}-1$ and hence by (\ref{eq1}), $e(G^r)\geq \binom{2k-3}{2}+3$. If instead $e(H_{2}^{r})=2$, i.e. $u_2v,u_3v\in E(G^r)$, then  $u_4u_2,u_4u_3\notin E(G)$, so $e(G)\leq \binom{n}{2}-2$, and again by (\ref{eq1}),  $e(G^r)\geq \binom{2k-3}{2}+4$.  

Since $G^r-D-EC(u_2u_3)$ contains no matching of size $k-1$, we have 
\begin{align*}
e(G^r-D)\leq ex(2k-2,(k-1)K_2)+1=\binom{2k-3}{2}+1.
\end{align*}
and therefore
\begin{align*}
 e(G^r)\leq e(H_{2}^{r})+e(G^{r}-D)\leq  e(H_{2}^{r})+\binom{2k-3}{2}+1,
\end{align*}
For $e(H_{2}^{r})=1$, this yields $e(G^r)\le\binom{2k-3}{2}+2$, and for $e(H_{2}^{r})=2$, we obtain $e(G^r)\le\binom{2k-3}{2}+3$. In both cases, the upper bound contradicts the corresponding lower bound, completing the proof.
\end{proof}

By Claim \ref{c19},  we  may assume  $S=\{v_1,\ldots,v_s\}$ with $2\le s \le k-2$.

\begin{claim}\label{c20}
	$\nu(H_2^r)=s$.
\end{claim}
\begin{proof}

By contradiction. Suppose that $H_2^r$ does not contain a  matching of size $s$. By Hall's theorem \cite{Hal}, there exists a nonempty subset $T\subseteq S$ such that $\left|N_{H_2^r}(T)\right|< \left|T\right|$, where $1\le \left|T\right|\le s$. Then 
\begin{align}\label{eq999}
	e(G^r)&=e(G^r[S\cup V(D_1)])+e(H_2^r)\notag\\
	&\le \binom{2k-s-1}{2}+\left|T\right|(\left|T\right|-1)+(s-\left|T\right|)(n-2k+s+1 ):=p(s,\left|T\right|).
\end{align}

If $3\le s\le k-2$,  then 
\[e(G^r)\le \max\{p(3,1),p(k-2,1),p(3,3),p(k-2,k-2)\}<ex(n,(k-1)K_2)+2,\]
contradicting (\ref{eq1}). Hence, $s=2$.

We first consider the situation when $|T|=1$.  Assume without loss of generality that $v_1\in S$ with $d_{H_2^r}(v_1)=0$. Since $s=2$ and $d_1=2k-5$, we have $|D|\ge 4$. Note that $e(G^r[V(D_1)\cup S])\le \binom{2k-3}{2}$ and $e(G^r)\ge \binom{2k-3}{2}+2$ by (\ref{eq1}). So $e_{G^r}(\left\{v_2\right\},D)\ge 2$.

We now show that there exists distinct indices $i,j\in\{2,\dots,q\}$ such that $u_iu_j\in E(G)$ and $e_{G^r}(\{v_2\},D\setminus \left\{u_i,u_j\right\})\ge 2$. For $e(G)=\binom{n}{2}$,  the assertion holds since $|D|\ge 4$ and  $e_{G^r}(\left\{v_2\right\},D)\ge 2$. So we next consider that $e(G)\le \binom{n}{2}-1$. Suppose for contradiction that for every  $u_iu_j\in E(G)$, we have $e_{G^r}(\{v_2\},D\setminus \left\{u_i,u_j\right\})\le 1$. Since $u_2u_3\in E(G)$, it follows that $e_{G^r}(\left\{v_2\right\},D)\le 3$ and hence $e(G^r)\le \binom{2k-3}{2}+3$. This forces $e(G)=\binom{n}{2}-1$ and  $e_{G^r}(\left\{v_2\right\},D)=3$. Let $u_i,u_j,u_k,u_l$ be distinct vertices in $D$ with  $u_iv_2,u_jv_2,u_kv_2\in E(G^r)$. Then by assumption, $u_iu_l,u_ju_l,u_ku_l\notin E(G)$, contradicting $e(G)=\binom{n}{2}-1$.  This establishes the assertion. 

Let $u_i,u_j,u_k,u_l$ be distinct vertices in $D$ with $u_iu_j\in E(G)$ and $v_2u_k,v_2u_l\in E(G^r)$. Note that $e(G^r)=e(G^r[V(D_1)\cup \left\{v_1\right\}])+e_{G^r}(\{v_2\},V(G)\setminus \{v_2\})$. Since $n\ge 12$, it follows from \eqref{eq1} that
\begin{align*}
e(G^r[V(D_1)\cup \left\{v_1\right\}])&\ge \max\left\{f_2(n,k),f_3(n,k)\right\}-e_{G^r}(\{v_2\},V(G)\setminus \{v_2\})\\
&\ge \max\left\{f_2(n,k),f_3(n,k)\right\}-(n-1)\\
&> ex(2k-4,(k-2)K_2)+1.
\end{align*}
Hence, $G^r[V(D_1)\cup \left\{v_1\right\}]$ contains a matching  $M_5$ of size $k-2$ that avoids color $C(u_iu_j)$. Then either $M_5\cup \left\{u_iu_j,v_2u_k\right\}$ or $M_5\cup \left\{u_iu_j,v_2u_l\right\}$ is a rainbow matching of size $k$ in $G$, a contradiction.

Next we consider the case $|T|=2$. By (\ref{eq999}), $	e(G^r)\le p(2,2)=\binom{2k-3}{2}+2$. 
Then  (\ref{eq1}) implies $e(G^r)=\binom{2k-3}{2}+2$ and $G\cong K_n$.  Since  $|N_{H_2^r}(T)|\le 1$,  there exists $u_i\in D$ such that $u_iv_1,u_iv_2\in E(G^r)$. Take $u_j,u_l\in D\setminus \{u_i\}$, then $$e(G^r[S\cup V(D_1)\cup \left\{u_i\right\}])-EC(u_ju_l)>ex(2k-2,(k-1)K_2),$$ 
so there exists a matching $M_6$ of size $k-1$ in $G^r[S\cup V(D_1)\cup \left\{u_i\right\}]$ avoiding color $C(u_ju_l)$. Then $M_6\cup \{u_ju_l\}$ is a rainbow matching of size $k$ in $G$, a contradiction.

In all cases, we reach a contradiction, proving the claim.
\end{proof}

Next, we establish upper and lower bounds on the number of edges in $G^r$ and $G$, respectively.
By Claim \ref{c18} and Theorem \ref{lem-Li}, we have $e(H^{r}_{2}-\{u_2,u_3\})\leq (s-1)(n-2k+s-1)$. Therefore, by Claim \ref{c19}, we have
\begin{align}\label{ine1}
	e(G^r)&= e(G^r[S])+e(D_1)+e_{G^r}(V(D_1)\cup\{u_2,u_3\},S)+e(H^{r}_{2}-\{u_2,u_3\})\notag\\
	&\leq \binom{s}{2}+\binom{2k-2s-1}{2}+s(2k-2s+1)+(s-1)(n-2k+s-1)\notag\\
	&\le \max\Big\{\binom{2k-5}{2}+n+2k-4,\binom{k-2}{2}+(k-3)(n-k-3)+5k-7\Big\}\\
	&\le ex(n,(k-1)K_2)+(n-2k+5)\label{eq22}.
\end{align}
By Claim \ref{c20},  let $M$ be a matching of size $s$ in $H^{r}_{2}$. Then by Claim \ref{c18},  all vertices in $V(H^{r}_{2})$ not covered by $M$ forms an independent set in $G$. It follows that 
\begin{align}\label{eq11}
e(G)\le \binom{n}{2}-\binom{|V(H^{r}_{2})|-2s}{2}=\binom{n}{2}-\binom{n+1-2k}{2}.
\end{align}
Combining inequalities (\ref{eq1}) and (\ref{eq11}), we obtain
\begin{align}\label{eq12}
e(G^r)\ge \max\{f_1(n,k),f_2(n,k),f_3(n,k)\}+\binom{n+1-2k}{2}.
\end{align}

By (\ref{eq1}) and (\ref{eq22}),
\begin{align}\label{eq0}
e(G)&\ge \binom{n}{2}+\max\{f_2(n,k),f_3(n,k)\}-e(G^r)=\binom{n}{2}-(n-2k+3).
\end{align}
Combining  (\ref{eq11}) and (\ref{eq0}), we can get $n\le 2k+3$ by calculating $n-2k+3\ge \binom{n+1-2k}{2}$.
Note that $n\ge 12$ and $2k+3\ge n\ge 2k+1$. We substitute $n=2k+1$, $n=2k+2$, and $n=2k+3$ into inequality (\ref{eq12}) respectively, and the calculations show that all lead to a contradiction with  (\ref{ine1})  (the detailed  calculation  can be found in  Appendix).

This completes the proof of Theorem \ref{main-thm}. \qed

\vspace{3mm}
\noindent\textbf{Remark}: The conclusion also holds for graphs of small order (i.e., $n \le 11$), with the exception of the case $n=2k=4$, for which the condition becomes $e(G) + c(G) \ge 10$.  However, as the proof is straightforward yet tedious to enumerate, we omit the details here.

\section*{Acknowledgement}
The authors would like to thank Binlong Li and Yang Meng for their interest in this work as well as their valuable discussions concerning Section 3.1.


\newpage
\section*{ Appendix}

In this section, we give a detailed calculation to show a contradiction between inequalities (\ref{eq12}) and (\ref{ine1})  for $2k+1\le n\le 2k+3$ and $n\ge 12$. Write 
\begin{align*}
p=\max\Big\{\binom{2k-5}{2}+n+2k-4,\binom{k-2}{2}+(k-3)(n-k-3)+5k-7\Big\}.
\end{align*}

\vspace{2mm}\textbf{Case 1.}~$n=2k+1$.
\vspace{2mm}

One can see that $f_1(2k+1,k)=2k^2-5k+2$, $f_2(2k+1,k)=2k^2-7k+8$, and $f_3(2k+1,k)=(3k^2-3k-2)/2$.
By a direct calculation,  we have 
\begin{align*}
\max\{f_1(2k+1,k),f_2(2k+1,k),f_3(2k+1,k)\}=
\begin{cases}
	f_3(2k+1,k),& \text{if $ k=6$;}\\
	f_1(2k+1,k),& \text{if $k\ge 7$.}
\end{cases}
\end{align*}
Thus, by (\ref{eq12}), we have
\begin{align*}
e(G^r)\ge\frac{3k^2-3k}{2}>\max\Big\{2k^2-7k+12,\frac{3k^2-5k+4}{2}\Big\} =p
\end{align*}
for $4\leq k\leq 6$; and  for $ k\geq 7$, we have
\begin{align*}
e(G^r)\ge 2k^2-5k+3> \max\Big\{2k^2-7k+12,\frac{3k^2-5k+4}{2}\Big\}=p,
\end{align*}
contradicting (\ref{ine1}).

\vspace{2mm}

\vspace{2mm}\textbf{Case 2.}~$n=2k+2$.
\vspace{2mm}

Similar to the proof of Case 1, we can get  
\begin{align*}
\max\{f_1(2k+2,k),f_2(2k+2,k),f_3(2k+2,k)\}=
\begin{cases}
	f_3(2k+2,k),& \text{if $5\leq k\leq 8$;}\\
	f_1(2k+2,k),& \text{if $k\ge 9$,}
\end{cases}
\end{align*}
where $f_1(2k+2,k)=2k^2-5k-1$, $f_2(2k+2,k)=2k^2-7k+8$, and $f_3(2k+2,k)=(3k^2-k-6)/2$.
Thus, by (\ref{eq12}),  we have
\begin{align*}
e(G^r)\ge\frac{3k^2-k}{2}>\max\Big\{2k^2-7k+13,\frac{3k^2-3k-2}{2}\Big\} =p
\end{align*}
for $4\leq k\leq 8$; and  for $ k\geq 9$,
\begin{align*}
e(G^r)\ge 2k^2-5k+2> \max\Big\{2k^2-7k+13,\frac{3k^2-3k-2}{2}\Big\}=p,
\end{align*}
contradicting  (\ref{ine1}).

\vspace{2mm}\textbf{Case 3.}~$n=2k+3$.
\vspace{2mm}

Similar to the proof of Case 1, we can get  
\begin{align*}
\max\{f_1(2k+3,k),f_2(2k+3,k),f_3(2k+3,k)\}=
\begin{cases}
	f_3(2k+3,k),& \text{if $5\leq k\leq 10$;}\\
	f_1(2k+3,k),& \text{if $k\ge 11$,}
\end{cases}
\end{align*}
where $f_1(2k+3,k)=2k^2-5k-5$, $f_2(2k+3,k)=2k^2-7k+8$, and $f_3(2k+3,k)=(3k^2+k-10)/2$.
Thus, by (\ref{eq12}),  we have 
\begin{align*}
e(G^r)\ge\frac{3k^2+k+2}{2}>\max\Big\{2k^2-7k+14,\frac{3k^2-k-8}{2}\Big\} =p
\end{align*}
for $4\leq k\leq 10$; and  for $ k\geq 11$, we have
\begin{align*}
e(G^r)\ge 2k^2-5k+1> \max\Big\{2k^2-7k+14,\frac{3k^2-k-8}{2}\Big\}=p,
\end{align*}
contradicting with (\ref{ine1}).

\end{document}